\documentclass[12pt,a4paper]{amsart}

\usepackage{latexsym} 

\usepackage{amssymb}
\usepackage{amsthm}
\usepackage{amsfonts}
\usepackage{amsmath}
\usepackage{amstext}
\usepackage{amscd}
\usepackage[dvips]{graphics}
\usepackage{epsfig}

\numberwithin{equation}{section}

\theoremstyle{plain}
\newtheorem{thm}{Theorem}[section] 
\newtheorem{prop}[thm]{Proposition}
\newtheorem{cor}[thm]{Corollary}
\newtheorem{lem}[thm]{Lemma}

\newtheorem{theorem*}{Theorem}[]

\theoremstyle{definition}
\newtheorem{defn}[thm]{Definition}

\newtheorem{example}[thm]{Example}

\theoremstyle{remark}
\newtheorem{rem}[thm]{Remark}

\newcommand{\N}{\mathbb{N}}
\newcommand{\R}{\mathbb{R}}

\newcommand{\Z}{\mathbb{Z}}

\newcommand{\stsim}{ \sim\raisebox{-1.ex}{\makebox[-13.pt]{\scriptsize{{st} }}} \ \ \  }

\title[Directional properties in o-minimal structures]
{Directional properties of sets definable\\
in o-minimal structures}

\author{Satoshi Koike, Ta L\^e Loi, Laurentiu Paunescu and Masahiro Shiota}
\address{Department of Mathematics, Hyogo University of Teacher Education,
Kato, Hyogo 673-1494, Japan}
\email{koike@hyogo-u.ac.jp} 
\address{Department of Mathematics, University of Dalat,
Dalat, Vietnam}
\email{taleloi@hotmail.com} 
\address{School of Mathematics, University of Sydney, Sydney, NSW, 2006,
Australia}
\email{laurentiu.paunescu@sydney.edu.au}
\address{Graduate School of Mathematics, Nagoya University,
Furo-cho, Chigusa-ku, Nagoya 464-8602, Japan}
\email{shiota@math.nagoya-u.ac.jp}


\date{9-2-2010}

\begin{document}

\thanks{This research is partially supported by the Grant-in-Aid 
for Scientific Research (No. 20540075) of Ministry of Education, 
Science and Culture of Japan, and HEM 21 Invitation Fellowship
Programs for Research in Hyogo.}

\maketitle

\begin{abstract}
In \cite{koikepaunescu} it was introduced the notion of  direction set
for a subset of $\R^n$,
and it was shown that the dimension of the common direction set of two
subanalytic subsets, called {\em directional dimension}, is preserved
by a bi-Lipschitz homeomorphism, provided that their images are 
also subanalytic.
In this paper we give a generalisation of the above result to sets 
definable in an o-minimal structure on an arbitrary real closed field. 
More precisely, we first prove our main theorem and discuss in detail
directional properties in the case of an Archimedean real closed field, 
and in \S 7 we give a proof 
in the  case of a general real closed field.
In addition, related to our main result, we show the existence
of special polyhedra in some Euclidean space, illustrating that the
bi-Lipschitz equivalence does not always imply the existence of 
a definable one.
\end{abstract}


\bigskip
\section{Introduction}\label{introduction}
\medskip

We first recall the notions of direction set and real tangent cone in $\R^n$.

\begin{defn}\label{directionset}
Let $A$ be a set-germ at $0 \in \R^n$ such that
$0 \in \overline{A}$.
We define the {\em direction set} $D(A)$ of $A$ at $0 \in \R^n$ by
$$
D(A) := \{a \in S^{n-1} \ | \
\exists  \{ x_i \} \subset A \setminus \{ 0 \} ,
\ x_i \to 0 \in \R^n  \ \text{s.t.} \
{x_i \over \| x_i \| } \to a, \ i \to \infty \}.
$$
Here $S^{n-1}$ denotes the unit sphere 
centred at $0 \in \R^n$.

We denote by $LD(A)$ a half-cone of $D(A)$ with the origin 
$0 \in \R^n$ as the vertex:
$$
LD(A) := \{ t a \in \R^n\ | \ a \in D(A), \ t \ge 0 \},
$$
and call it the {\em real tangent cone} of $A$ at $0 \in \R^n$. 
\end{defn}

Let us examine an example.

\begin{example}\label{example11}
Let $h : \R^3 \to \R^3$ be a semialgebraic homeomorphism defined by 
$h(x,y,z) = (x,y,z^3)$, and let
$V = \{ (x,y,z) \in \R^3 : x^2 + y^2 - z^6 = 0 \}$. 
Then $V$ and $h(V) = \{ (x,y,z) \in \R^3 : x^2 + y^2 - z^2 = 0 \}$ 
are algebraic sets. 
It is easy to see that $\dim D(A) = 0$ and $\dim D(h(A)) = 1$. 
Therefore the dimension of direction sets is 
not a homeomorphic invariant.
\end{example}

We next investigate whether the dimension of direction sets 
is a Lipschitz invariant.
There are many singular examples of bi-Lipschitz homeomorphisms.
In \cite{koikepaunescu} it was given an example of a 
``quick spiral bi-Lipschitz
homeomorphism" and of a ``zigzag bi-Lipschitz homeomorphism".
Here we give a different one.

\begin{example}\label{example12}(Oscillation). 
Let $h : (\R^2,0) \to (\R^2,0)$ be a mapping defined by
$h(x,y)=(x, y + f(x))$, where $f(x) = x \sin(\ln |x|)$, 
and let $A = \R \times 0$. 
Then we can see that $h$ is a bi-Lipschitz homeomorphism 
and $h(A)$ is the graph of $f$. 
In addition, we have $\dim D(A) = 0$ and $ \ \dim(D(h(A))) = 1$.
Consequently the dimension of direction sets is not 
a bi-Lipschitz invariant either.
\end{example}

Note that in the above example $h(A)$ is not a subanalytic set.
Therefore we may ask whether the dimension of direction
sets is a bi-Lipschitz invariant in the case when the image $h(A)$ 
is also subanalytic.
In fact, the following result is shown in \cite{koikepaunescu}.

\begin{thm}\label{maintheorem}
(Main Theorem in \cite{koikepaunescu})
Let $A$, $B \subset \R^n$ be subanalytic set-germs at $0 \in \R^n$ 
such that $0 \in \overline{A} \cap \overline{B}$, and let 
$h : (\R^n,0) \to (\R^n,0)$ be a bi-Lipschitz homeomorphism.
Suppose that $h(A), \ h(B)$ are also subanalytic.
Then we have
$$
\dim (D(h(A)) \cap D(h(B))) = \dim (D(A) \cap D(B)).
$$
\end{thm}
See H. Hironaka \cite{hironaka} for subanalyticity.

A subfield of $\R$ is called {\it Archimedean}.
This terminology comes from the fact that an ordered field $R$ 
is isomorphic to a subfield of $\R$ if and only if for any positive 
elements $a$ and $b$ of $R$ there exists a natural number $n$ 
such that $a n > b$. 
An ordered field $R$ is called {\it real closed} if its 
complexification $R[t]/(1+t^2)R[t]$ is algebraically closed. 
The field of real numbers $\R$ is an Archimedean real closed field.
Another well-known example of Archimedean real closed field
is the field of real algebraic numbers. 
In this paper we give a generalization of Theorem \ref{maintheorem} 
to the case of sets definable in an arbitrary o-minimal structure 
on an arbitrary real closed field (Theorem \ref{non-Archimedean}).
We first show the main theorem for the case of an Archimedean
real closed field.
Namely, we show the following:

\begin{thm}\label{omaintheorem}
Let $R$ be an Archimedean real closed field, and
let $A, \ B$ be definable set-germs at $0$ in $R^n$ 
in an o-minimal structure on $R$
such that $0 \in \overline{A} \cap \overline{B}$. 
Let $h : (R^n,0) \to (R^n,0)$ be a bi-Lipschitz homeomorphism. 
Suppose that $h(A), \ h(B)$ are also definable. 
Then we have
$$
\dim (D(h(A)) \cap D(h(B))) = \dim (D(A)\cap D(B)).
$$
\end{thm}
See the next section for the definition of a definable set
and of the direction set $D(A)$ of a set $A$ in $R^n$.

Theorem \ref{maintheorem} was shown using essentially the following 
ingredients:

\vspace{3mm}

\qquad (1) Sea-tangle properties;

\qquad (2) Sequence selection properties;

\qquad (3) Volume arguments.

\vspace{3mm}

In \S 2 we describe the notion of o-minimal structure
and point out some of its properties.
We give in \S 3 an important example concerning the relationship
between  bi-Lipschitz equivalence and  definable equivalence.
After this we introduce an adapted notion of sea-tangle neighbourhood,
using ordered definable functions and describe several
of its properties
 in \S 4.
In \S 5 we discuss sequence selection properties,
and we give the proof of our main theorem
(Theorem \ref{omaintheorem}) using volume arguments in \S 6.
In \S \S 4 - 6 we develop the arguments in any o-minimal
structure on any Archimedean real closed field.
In \S 7 we generalise the main theorem to any real closed field
and give a proof. 
This proof is rather resembling  proofs in  Logic Theory. 
Essentially,  this shows that in our proof 
of the main theorem we do not use all special properties 
of the real number field (e.g. local compactness).

\bigskip
\section{o-minimal structure}\label{ominimal}
\medskip

Throughout this paper, except \S 7, $R$ denotes an Archimedean real 
closed field.

Concerning the direction set, let us analyse the following example:

\begin{example}\label{example21}
Let $R$ be the field of real algebraic numbers, and let
$\{ a_m \}$ be the sequence of points of $R^2$ defined by
$$
a_m = ({1 \over m}, {1 \over m}(1 + {1 \over 1!} +
{1 \over 2!} + \cdots + {1 \over m!})).
$$
Clearly  $a_m$ tends to $0 \in R^2$, and 
${ a_m \over \| a_m \| }$ tends to a pair of transcendental numbers
$({1 \over \sqrt{1 + e^2}}, {e \over \sqrt{1 + e^2}})$
which is not an element of $R^2$.

We can see that for any $p \in S^1 \subset \R^2$,
there is a sequence of points $\{ a_m \}$ of $R^2$ tending to $0 \in R^2$
such that ${ a_m \over \| a_m \| }$ tends to $p$.
Therefore we have $\dim_R R^2 = 2$, but there are infinitely 
many points in $S^1 \setminus R^2$ which are independent over $R$.

Let $\{ b_m \}$ be the sequence of points of $R^2$ defined by
$$
b_m = (0, {1 \over m}(1 + {1 \over 1!} + {1 \over 2!} + \cdots 
+ {1 \over m!})).
$$
Then we can see that there is a bi-Lipschitz homeomorphism
$h : (R^2,0) \to (R^2,0)$ such that for all $m$, $h(a_m) = b_m$.
On the other hand we have 
$$
\lim_{m \to \infty} { a_m \over \| a_m \| } \notin R^2, \ \
\lim_{m \to \infty} { b_m \over \| b_m \| } \in R^2.
$$
\end{example}

To avoid any confusion when we consider the limit, 
we give the precise definition of the direction set 
over an Archimedean real closed field $R$.

\begin{defn}\label{rcfdirectionset}
Let $A$ be a set-germ at $0 \in R^n$ such that
$0 \in \overline{A}$.
We define the {\em direction set} $D(A)$ of $A$ at $0 \in R^n$ by
$$
D(A) := \{a \in S^{n-1} \ | \
\exists  \{ x_i \} \subset A \setminus \{ 0 \} ,
\ x_i \to 0 \in R^n  \ \text{s.t.} \
{x_i \over \| x_i \| } \to a, \ i \to \infty \}.
$$
Here $S^{n-1} \subset R^n$ denotes the unit sphere 
centred at $0 \in R^n$.

We denote by $LD(A)$ a half-cone of $D(A)$ with the origin 
$0 \in R^n$ as the vertex:
$$
LD(A) := \{ t a \in R^n\ | \ a \in D(A), \ t \ge 0 \}.
$$ 
\end{defn}

Let us recall the definition of an o-minimal structure
on a real closed field $R$.

\begin{defn}\label{structure}
Let $\mathcal{D}$ be a sequence $(\mathcal{D}_n)_{n \in \N}$
where for each $n \in \N$, $\mathcal{D}_n$ is a family of
subsets of $R^n$. 
We say that $\mathcal{D}$ is an {\em o-minimal structure}
on $R$ if:

(D1) $\mathcal{D}_n$ is a boolean algebra.

(D2) If $A \in \mathcal{D}_n$, then $A \times R$ and
$R \times A \in \mathcal{D}_{n+1}$.

(D3) If $A \in \mathcal{D}_{n+1}$, then $\pi (A) \in \mathcal{D}_n$,
where $\pi : R^{n+1} \to R^n$ is the projection on the first
$n$ coordinates.

(D4) $\mathcal{D}_n$ contains $\{ x \in R^n : P(x) = 0 \}$
for every polynomial $P \in R [X_1, \cdots , X_n]$.

(D5) Each set in $\mathcal{D}_1$ is a finite union of
intervals and points.
\end{defn}

A subset $A$ of $R^n$ belonging to $\mathcal{D}_n$ is called 
{\em definable} in $\mathcal{D}$.
A map $f : A \to R^m$ is {\em definable} in $\mathcal{D}$,
if its graph is a definable subset of $R^n \times R^m$ 
in $\mathcal{D}$.

The class of semi-algebraic sets and the class of global 
sub-analytic sets are examples of o-minimal structures
on the field of real numbers $(\R,+,\cdot)$.
We refer the readers to \cite{dries}, \cite{driesmiller} and 
\cite{coste} for the basic properties of o-minimal structures. 
In particular, the following results and properties are frequently 
used in our paper.

\vspace{3mm}

(1) The dimension of definable sets is well-defined (this follows from the Cell Decomposition Theorem
(\cite{dries}, Chapter 3 (2.11)));

(2) Monotonicity (\cite{dries}, Chapter 3 (1.2));

(3) Curve Selection Lemma (\cite{dries}, Chapter 6 (1.5)). 

\vspace{3mm}

We mention one more fact.
The Lojasiewicz inequalities for definable sets in an
o-minimal structure are discussed in \cite{driesmiller} and
\cite{taleloi}.
The Lojasiewicz inequalities in the sense of \cite{lojasiewicz}
hold in {\em polynomially bounded} o-minimal structures.

Before ending this section, we introduce a notation 
which we will often use in this paper, namely we denote by
$\Phi$  the set of all odd, strictly increasing,
continuous definable germs from $(R,0)$ to $(R,0)$. 
Note that, by Monotonicity, $\Phi$ is ordered by the following relation:

\vspace{3mm}

\centerline{$\theta_1 \leq \theta_2$ iff
$\theta_1(t) \leq \theta_2(t)$, for all $t > 0$ sufficiently small.}

\vspace{3mm}

\bigskip
\section{Bi-Lipschitz equivalence does not always imply definable one}
\label{definablebilipschitz}
\medskip

In our main theorem (Theorem (\ref{omaintheorem})) we did not assume the 
definability of the Lipschitz homeomorphism $h : (R^n,0) \to (R^n,0)$.
In the case when  $h$ is definable, we can easily show 
the theorem as follows.

Let $\overline{h} : S^{n-1} \to S^{n-1}$, $S^{n-1} \subset R^n$,
be a mapping defined by
$$
\overline{h}(a) = 
\lim_{t \to 0} {h(ta) \over \| h(ta) \|}.
$$
Then it is easy to see that $\overline{h}$ is a well-defined 
 definable bi-Lipschitz homeomorphism.
Therefore it follows that $\overline{h}(D(A)) = D(h(A)))$.

Related to the above fact, it may be natural to ask if we can replace
a bi-Lipschitz homeomorphism $h$ with a definable 
bi-Lipschitz homeomorphism $h^{\prime}$. That is to say, whether
the existence of $h$ implies the existence of a definable 
bi-Lipschitz homeomorphism $h^{\prime}$ with
$h^{\prime}(A) = h(A)$ and $h^{\prime}(B) = h(B)$.
If the answer were positive, we would have a different proof
of our main theorem, without using the  main properties
mentioned in the introduction.
Nevertheless, the answer to this question is negative.
More precisely, bi-Lipschitz equivalence does not always
guarantee the existence of  a definable one.

\begin{thm}\label{bilipdefinable}
There exist $n\in\N$ and compact polyhedra $A_1$ and $A_2$ in $\R^n$,  
such that the germs of $(\R^n,A_1)$ and $(\R^n,A_2)$ 
at $0 \in \R^n$ are bi-Lipschitz homeomorphic but not definably 
homeomorphic in any o-minimal structure on $\R$. 
\end{thm}

\begin{proof}
We first recall the following result of R. C. Kirby and 
L. C. Siebenmann \cite{kirbysiebenmann}.

\vspace{3mm}

{\em For a PL manifold $X_1$ of dimension $\ge 5$ with 
$H^3(X_1;Z_2) \not= 0$, there exists a PL manifold $X_2$ which is 
homeomorphic but not PL homeomorphic to $X_1$.}

\vspace{3mm}

\noindent Let $X_1$ and $X_2$ be such compact manifolds contained 
in $\R^{m_1}$ and $\R^{m_2}$, respectively, and let $h : X_1 \to X_2$
be a homeomorphism. 
On the other hand, by D. Sullivan \cite{sullivan}, the Lipschitz 
manifold structure on a topological manifold of dimension $\not=4$ 
is unique up to bi-Lipschitz homeomorphisms. 
Therefore we can choose $h$ as a bi-Lipschitz homeomorphism
since a PL manifold is a Lipschitz manifold. 
For a point $x$ in $\R^n$ and a subset $X$ of $\R^n$, 
let $x*X$ denote the cone with vertex $x$ and base $X$,
and let $X_x$ be the germ of $X$ at $x$. 
Set 
\begin{eqnarray*}
& Y_i := & 0*(X_i \times \{ 1 \} ) \subset \R^{m_i} \times \R \quad 
\text{for}\ i = 1,2, \\
& A_1 := & Y_1 \times \{ 0 \} \subset \R^{m_1 +1} \times \R^{m_2 +1}, \\
& A_2 := & \{ 0 \} \times Y_2 \subset \R^{m_1 +1} \times \R^{m_2 +1}, \\
& n := & m_1 + m_2 + 2.
\end{eqnarray*}

We first show the following claim.

\vspace{3mm}

\noindent {\em Claim 1}. The following germs at $0$, 
$(\R^n,A_1)_0$ and $(\R^n,A_2)_0$, are bi-Lipschitz
homeomorphic.

\begin{proof} 
The idea of our proof comes from the proof of Proposition 10.4 
in R. J. Daverman \cite{daverman}. 
First we extend 
$h \times \text{id}\ : X_1 \times \{ 1 \} \to X_2 \times \{ 1 \}$ 
to a bi-Lipschitz homeomorphism $h^* : Y_1 \to Y_2$ by cone extension. 
To be precise, set $h^*(0) := 0$ and 
$$
h^*(tx,t) := (th(x),t) \quad \text{for}\ (t,x) \in (0,\,1) \times X_1.
$$
Then $h^*$ is bijective and $(h^*)^{-1}(tx,t)=(th^{-1}(x),t)$.
Moreover, we can see that $h^*$ is Lipschitz as follows. 
Let $(t,x),(t',x') \in [0,\,1] \times X_1$. 
Then we have
\begin{eqnarray*}
\| h^*(tx,t) - h^*(t'x',t')\| \le \| th(x) - t'h(x')\| + \| t - t'\|, \\
\| th(x) - t'h(x')\| \le \| th(x) - t'h(x)\| + \| t'h(x) - t'h(x')\| \\
\le c\| t - t'\| + ct'\| x - x'\|, \\
t'\| x - x'\| \le \| tx - t'x\| + \|tx - t'x'\| \le c\| t - t'\| 
+ \| tx - t'x'\|
\end{eqnarray*}
for some constant real number $c > 0$. 
Hence we have
$$
\| h^*(tx,t)-h^*(t'x',t')\| \le c'\| t - t'\| + c'\| tx - t'x'\|
$$
for some constant real number $c' > 0$. 
In the same way we can see that $(h^*)^{-1}$ is Lipschitz. 
Thus $h^*$ is bi-Lipschitz.

Secondly we extend $h^*$ to a Lipschitz (not bi-Lipschitz) map 
$\tilde{h} : \R^{m_1+1} \to \R^{m_2+1}$. 
Let $K$ be a simplicial decomposition of $\R^{m_1+1}$ such that $Y_1$ 
is the underlying polyhedron of a full subcomplex $K_1$ 
of $K$.  $K_1$ is called {\em full} in $K$ if each simplex 
in $K$ with all its vertices  in $K_1$ is necessarily contained in $K_1$. 
(When $K_1$ is not full in $K$, we replace $K$ and $K_1$ with their 
barycentric subdivisions $K'$ and $K'_1$.
Then $K'_1$ is full in $K'$. 
See C. P. Rourke and B. J. Sanderson \cite{rourkesanderson}
for  details.) 
Let $K^r$ denote the $r$-skeleton of $K$, namely, the simplexes 
in $K$ of dimension $\le r$. 
We define $\tilde{h}$ on the underlying polyhedron $|K^r|$ of $K^r$ 
by induction on $r$. 
If $r = 0$, set $\tilde{h} := 0$ on $|K^0| - Y_1$. 
Assume that $\tilde{h}$ is already defined on $|K^{r-1}|$ for some $r > 0$. 
For each $\sigma \in K^r-K^{r-1}$ with $\sigma \not\subset Y_1$, 
let $v_0, \cdots ,v_r$ be the vertices of $\sigma$ such that 
$v_0 \not\in Y_1$, which exists by the fullness of $K_1$. 
Note that $v_0,v_1* \cdots * v_r \in K^{r-1}$. 
Set
\begin{eqnarray*}
\tilde{h}(\sum_{i=0}^r t_i v_i) := \sum_{k=1}^r t_k \tilde{h}(\sum_{i=1}^r 
t_i v_i / \sum_{j=1}^r t_j) \\
\text{for}\ (t_0,...,t_r) \in [0,\,1]^{r+1}\ \text{with}\ \sum_{i=0}^r t_i
= 1 \ \text{and}\ \sum_{i=1}^r t_i \not= 0,
\end{eqnarray*}
which is well-defined because 
$\sum_{i=1}^r t_i v_i / \sum_{j=1}^r t_j\in v_1* \cdots * v_r$. 
Then $\tilde{h}$ is a map from $|K^r|$ to $\R^{m_2+1}$ and we claim that it
 is 
Lipschitz.
In order to see this, it suffices to show that $\tilde{h}|_\sigma$ is 
Lipschitz for the above $\sigma$, because $\tilde{h} = 0$ 
outside of a compact neighbourhood of $Y_1$ in $\R^{m_1+1}$. 
By the above definition of $\tilde{h}|_\sigma$, 
$\tilde{h}|_\sigma(v_0) = 0$ and $\tilde{h}|_\sigma$ is the cone 
extension of $\tilde{h}|_{v_1* \cdots * v_r}$. 
Hence, as shown above, $\tilde{h}|_\sigma$ is Lipschitz 
since so is $\tilde{h}|_{v_1*\cdots* v_r}$.

Set $A_3 := \text{graph}\ h^ * \subset \R^{m_1+1} \times \R^{m_2+1}$. 
We shall prove that $(\R^n,A_3)_0$ is bi-Lipschitz homeomorphic 
to $(\R^n,A_1)_0$ and $(\R^n,A_2)_0$. 
Consider $(\R^n,A_3)_0$ and $(\R^n,A_1)_0$. 
Set 
$$
\phi(x,y) := (x,y - \tilde{h}(x)) \quad \text{for}\ (x,y) \in 
\R^{m_1+1} \times \R^{m_2+1}. 
$$ 
Then $\phi$ is a homeomorphism of $\R^n$, $\phi^{-1}(A_1) = A_3$, 
$\phi(0) = 0$, and $\phi$ is Lipschitz since so is $\tilde{h}$. 
In addition, $\phi$ has its inverse  given by
$$
\R^{m_1+1} \times \R^{m_2+1 }\ni (x,y) \to (x,y + \tilde{h}(x))
\in \R^{m_1+1} \times \R^{m_2+1},
$$ 
which is also Lipschitz. 
Thus $\phi$ is a bi-Lipschitz homeomorphism from $(\R^n,A_3)_0$ 
to $(\R^n,A_1)_0$. 
In the same way, by extending $(h^*)^{-1}$ to a Lipschitz map 
from $\R^{m_2+1}$ to $\R^{m_1+1}$, we see that $(\R^n,A_3)_0$ and 
$(\R^n,A_2)_0$ are bi-Lipschitz homeomorphic.
\end{proof}

We next show the following claim.

\vspace{3mm}

\noindent {\em Claim 2}. $(\R^n,A_1)_0$ and $(\R^n,A_2)_0$ are not 
definably homeomorphic.

\begin{proof} 
Assume that they are definably homeomorphic. 
Then $(A_1)_0$ and $(A_2)_0$  are definably homeomorphic. 
Hence shrinking $A_1$ we have a definable embedding 
$f : A_1 \to A_2$ such that $f(0) = 0$ and $f(A_1)$ is 
a neighbourhood of 0 in $A_2$. 

Here we recall the following facts from \cite{shiota}.

\vspace{3mm}

\noindent {\bf Definable Triangulation Theorem.}
(Theorem II.2.1 in \cite{shiota})
{\em Any compact definable set $X$ is definably homeomorphic to 
some polyhedron $X'$.}

\begin{rem}\label{remark01} (Remark II.2.3 in \cite{shiota})
If the above $X$ is contained in the underlying polyhedron 
of a finite simplicial complex $K$,
then we can choose $X'$ in $|K|$ and a definable homeomorphism 
$g : X \to X'$ so that $g(X \cap |\sigma |) \subset |\sigma |$ 
for each $\sigma\in K$.
\end{rem}

\noindent {\bf Definable Hauptvermutung.} (Corollary III.1.4 in \cite{shiota})
{\em Any two compact polyhedra are PL homeomorphic 
if they are definably homeomorphic.}

\vspace{3mm}

Applying the definable triangulation theorem and Remark \ref{remark01} 
to $f(A_1)$ and 0, we can assume that $f(A_1)$ is a polyhedron.
Therefore we may assume $f(A_1) = A_2$ from the beginning. 
The 
links of $A_1$ and $A_2$ at 0 are $X_1$ and $X_2$ respectively. 
By our choice of $X_1$, they are not PL homeomorphic to  a sphere or to 
a ball, hence $A_1$ and $A_2$ are not PL manifolds at $0$ (0 is a 
singular point).
On the other hand the links at the other points are all PL homeomorphic 
to a sphere or to a ball. 
By the definable Hauptvermutung, $A_1$ and $A_2$ are PL homeomorphic as 
polyhedra. 
Moreover, because the origin 0 is 
the only  singular point of $A_1$ and $A_2$, the PL homeomorphism has to 
carry 0 to 0.
Thus $(A_1)_0$ and $(A_2)_0$ 
are PL homeomorphic, which is a contradiction because of our 
our choice of $X_1$ and $X_2$. 
\end{proof}

This completes the proof of the theorem.

\end{proof}

\begin{rem}\label{remark02}
In the proof of Claim 1 we constructed a Lipschitz extension
of $h^*$ to $\tilde{h}$ using the cone structure.  
However, in general, to extend a Lipschitz map is not difficult. 
Indeed, for a Lipschitz function with constant $L, \
f : A \to \R, \ A \subset X$, $A$ endowed with the induced 
metric from $(X,d)$, we have an extension formula 
(see S. Banach \cite{banach}):
$$
\alpha(x):=\inf_{a\in A}(f(a)+Ld(x,a)).
$$
Similarly one can extend it by 
$$
\beta(x):=\sup_{a\in A}(f(a)-Ld(x,a)).
$$ 
Note that $\beta(x)\leq \alpha(x)$.
Any convex combination $t\alpha(x)+(1-t)\beta(x)$, $0 \leq t \leq 1$, 
also gives a Lipschitz extension.

This construction can be used to extend Lipschitz maps as well, however, 
without preserving the Lipschitz constant.
\end{rem}

\bigskip
\section{Sea-Tangle Properties in o-minimal Structures}
\label{seatangle}
\medskip

We recall the notion of sea-tangle neighbourhood for a
subset of $\R^n$, originated from the classical notion of 
{\em horn-neighbourhood} for an analytic set or more generally
a subanalytic set in $\R^n$.

\begin{defn}\label{seatanglenbd}
Let $A \subset \R^n$ such that $0 \in \overline{A}$,
and let $d, \ C > 0$.
The {\em sea-tangle neighbourhood $ST_d(A;C)$ of $A$,
of degree $d$ and  width} $C$, is defined by:
$$
ST_d(A;C) := \{ x \in \R^n \ | \ dist (x,A) \le C \| x \|^d \}.
$$
\end{defn}

See \S 4 of \cite{koikepaunescu} for  some
sea-tangle properties.
For instance,  the following is shown.

\begin{prop}\label{keyproperty}(\cite{koikepaunescu} Proposition 4.7)
Let $A$ be a subanalytic set-germ at $0 \in \R^n$ such that 
$0 \in \overline{A}$. 
Then there is $d_1 > 1$ such that $A \subset ST_d(LD(A);C)$ as set-germs 
at $0 \in \R^n$ for any $d$ with $1 < d  < d_1$ and $C > 0$.
\end{prop}

Both Hironaka's selection lemma for  subanalytic sets and a Lojasiewicz 
inequality have played an  important role in the proof of the above result.
As mentioned in \S 2, it is known that 
the curve selection lemma holds also for definable sets 
in an o-minimal structure; on the other hand, the usual Lojasiewicz inequality 
does not always hold in an o-minimal structure.
Accordingly, Proposition \ref{keyproperty} might be false
for a definable set in some o-minimal structure.
Indeed, the following example confirms this.
\begin{example}\label{curve}
Let $\pi : \mathcal{M}_2 \to \R^2$ be a blowing-up at $(0,0) \in \R^2$,
and let $a = (0,1) \in S^1$.
We denote by $L(a)$ the half line in $\R^2$ with the origin
as the starting point passing through $a$ and by $\hat{L}(a)$ 
the strict transform of $L(a)$ in $\mathcal{M}_2$ by $\pi$.
In a suitable coordinate neighbourhood,
$\pi : \R^2_{(X,Y)} \to \R^2$ can be expressed as $\pi (X,Y) = (XY,Y)$.
Here $(0,0) \in \R^2_{(X,Y)}$ is the intersection of $\hat{L}(a)$
and the exceptional divisor $E = \pi^{-1}(0,0)$.

Let
$B := \{ (X,Y)\in \R^2_{(X,Y)} \ | \ Y = e^{- {1 \over |X|^2}},
\ X \ge 0 \}$.
Then the curve $B$ is not contained in
$\{ (X,Y) \in \R^2_{(X,Y)} \ | \ 
|Y| \ge C^{\prime} |X|^{d^{\prime}} \}$
as germs at $(0,0) \in \R^2_{(X,Y)}$,
for any $d^{\prime} > 0$, $C^{\prime} > 0$.

Let $\R_{exp}$ be Wilkie's exponential field (\cite{wilkie}),
and let $\mathcal{D}$ be the o-minimal structure on it.
Set $A := \pi (B)$. 
Then we can see that $A \in \mathcal{D}$ and $LD(A) = L(a)$, 
but $A$ is not contained in any sea-tangle neighbourhood $ST_d(L(a);C)$
as germs at $(0,0) \in \R^2$, for $d > 1$, $C > 0$. 
Therefore Proposition \ref{keyproperty} does not hold
for this definable set $A$ in the o-minimal structure $\mathcal{D}$.
\end{example}

Taking into account the above fact, in order to develop
sea-tangle properties in an o-minimal structure
on an Archimedean real closed field $R$, the definition of 
sea-tangle neighbourhood has to be modified.
Note that $R^n$ has an induced metric from $\R^n$.
>From now on let us fix an o-minimal structure on $R$.
Here we recall that $\Phi$ is the set of all odd, strictly increasing,
continuous definable germs from $(R,0)$ to $(R,0)$.
Then we  define the notion of sea-tangle neighbourhood
of a definable set as follows:

\begin{defn}\label{oseatanglenbd}
Let $A \subset R^n$ such that $0 \in \overline{A}$,
and let $\theta \in \Phi$.
The {\em sea-tangle neighbourhood $ST_{\theta}(A)$ of $A$
with respect to} $\theta$ is defined by:
$$
ST_{\theta}(A) := \{ x \in R^n \ | \ dist (x,A) \le 
\theta (\| x \| ) \| x \| \}.
$$
\end{defn}

\begin{rem}\label{remark30} (1) Let $x \in R^n$ and $A \subset R^n$.
In general, dist $(x,A) = \inf_{a\in A} d(x,a)$ does not 
always belong to $R$; nonetheless it is always a non-negative real number.

(2) If $A$ is definable, then $D(A)$, $LD(A)$ and $ST_{\theta}(A)$
are also definable.
\end{rem}

Let $\mathcal{S}$ be the set of set-germs $A \subset R^n$
at $0 \in R^n$ such that $0 \in \overline{A}$. 

\begin{defn}\label{oSTequivalence}
Let $A$, $B \in \mathcal{S}$.
We say that $A$ and $B$ are $ST$-{\em equivalent},
if there are $\theta_1, \ \theta_2 \in \Phi$ 
such that $B \subset ST_{\theta_1}(A)$ and $A \subset ST_{\theta_2}(B)$
as germs at $0 \in R^n$. 
We write $A \stsim B$.
\end{defn}

\begin{rem}\label{remark31}
$ST$-equivalence $\stsim$ is an equivalence relation in $\mathcal{S}$.
\end{rem}

We first describe several sea-tangle properties for  general subsets
of $R^n$.

Let $\phi : (R^n,0) \to (R^n,0)$ be a bi-Lipschitz homeomorphism,
namely there are positive numbers $K_1, \ K_2 \in R$
with $0 < K_1 \le K_2$ such that
$$
K_1 \| x_1 - x_2 \| \le \| \phi (x_1) - \phi (x_2) \| 
\le K_2 \| x_1 - x_2 \|
$$
in a small neighbourhood of $0 \in R^n$.
Conversely, we have
$$
\frac{1}{K_2} \| y_1 - y_2 \| \le \| \phi^{-1}(y_1) - \phi^{-1}(y_2) \| 
\le \frac{1}{K_1} \| y_1 - y_2 \|
$$
in a small neighbourhood of $0 \in R^n$.
With these Lipschitz constants  we can formulate the following Sandwich 
Lemma.

\begin{lem}\label{osandwichlemma}
Let $A \subset R^n$ such that $0 \in \overline{A}$.
Then, for $\theta \in \Phi$, 

\vspace{3mm}

(i) $\phi(ST_{\theta}(A)) \subset ST_{\theta_1}(\phi(A))$
where $\theta_1(t) = {K_2 \over K_1} \theta ({t \over K_1})
\in \Phi$. and

\vspace{3mm}

(ii) $ST_{\theta_2}(\phi(A)) \subset \phi(ST_{\theta}(A))$
where $\theta_2(t) = {K_1 \over K_2} \theta ({t \over K_2})
\in \Phi$

\vspace{3mm}

\noindent in a small neighbourhood of $0 \in R^n$.
\end{lem}

The next proposition follows from the above Sandwich Lemma.

\begin{prop}\label{ST-equivalence}
$ST$-equivalence is preserved by a bi-Lipschitz homeomorphism.
\end{prop}

In \cite{koikepaunescu} there are mentioned  some directional properties
for the original notion of sea-tangle neighbourhood $ST_d(A;C)$.
The same properties hold also for our sea-tangle neighbourhood
$ST_{\theta}(A)$.
Throughout this section,
let $A$, $B \subset R^n$ be set-germs at $0 \in R^n$
such that $0 \in \overline{A} \cap \overline{B}$,
namely $A$, $B \in \mathcal{S}$,
and let $h : (R^n,0) \to (R^n,0)$ be a bi-Lipschitz homeomorphism.

\begin{lem}\label{keygeneral}
Suppose that there is $\theta \in \Phi$ such that $A \subset ST_{\theta}(B)$ 
as set-germs at $0 \in R^n$.
Then we have $D(h(A)) \subset D(h(B))$.
In addition, we have $D(ST_{\theta^{\prime}}(h(A))) \subset D(h(B))$
for any $\theta^{\prime} \in \Phi$.
\end{lem}

\begin{proof}
Let $a \in D(h(A))$.
Then there is a sequence of points $\{ a_m \} \subset A$ tending to
$0 \in R^n$ such that 
$\lim_{m \to \infty} {h(a_m ) \over \| h(a_m )\|} = a$.
By assumption, $A \subset ST_{\theta}(B)$.
Therefore, for each $m$ we can take $b_m \in B$ such that
$$
\| a_m - b_m \| \ll \| a_m \|, \ \| b_m \| .
$$
Since $h$ is a bi-Lipschitz homeomorphism,
$$
\| h(a_m ) - h(b_m )\| \ll \| h(a_m )\|, \ \| h(b_m )\| .
$$
Thus we have
$$
a = \lim_{m \to \infty} {h(a_m ) \over \| h(a_m )\|}
= \lim_{m \to \infty} {h(b_m ) \over \| h(b_m )\|} \in D(h(B)).
$$

The second statement follows from the first one.
\end{proof}

We have the following corollary.

\begin{cor}\label{cor1}
$(1)$ $D(ST_{\theta}(A)) = D(A)$ for any $\theta \in \Phi$.

$(2)$ $D(ST_{\theta}(h(A))) = D(h(A))$ for any $\theta \in \Phi$.

$(3)$ If $A$ and $B$ are $ST$-equivalent,
then we have $D(A) = D(B)$ and $D(h(A)) = D(h(B))$.
\end{cor}

For a definable set we have more specific sea-tangle properties.
Proposition \ref{keyproperty} is modified to the following.

\begin{prop}\label{okeyproperty}
Let $A$ be a definable set-germ at $0 \in R^n$ such that 
$0 \in \overline{A}$. 
Then there is $\theta \in \Phi$
such that $A \subset ST_{\theta}(LD(A))$ as set-germs at $0 \in R^n$.
\end{prop}

\begin{proof}
Let $g(x) = {d(x,LD(A)) \over \| x \|}$ for $x \in \overline{A}$, 
$x \neq 0$.

\vspace{3mm}

\noindent {\em Claim}. $\lim_{x\to 0} g(x) = 0$.

\begin{proof} 
If the claim does not hold, then by the Curve selection lemma, 
there exist $c > 0$ and a definable curve $\gamma : (0,1) \to
\overline{A} \setminus \{ 0 \}$ such that $\lim_{t \to 0} \gamma(t) = 0$
and ${d(\gamma(t), LD(A)) \over \| \gamma (t)\| } \geq c$
for sufficiently small $t > 0$.
By Monotonicity,
$\lim_{t \to 0} {\gamma(t) \over \| \gamma (t) \|} = a \in D(A)$.
Therefore we have 
$$
{d(\gamma(t), LD(A)) \over \| \gamma (t) \|} \leq
{d(\gamma(t), L(a)) \over \| \gamma (t) \|} \to 0
$$ 
when $t\to 0$.
This is a contradiction.
\end{proof}

By this claim, $g$ can be naturally  extended to a continuous 
definable function $g : \overline{A} \to R$ with $g(0)=0$. 
By the Lojasiewicz inequality (\cite{driesmiller}), there exists
$\theta \in \Phi$ such that $g(x) \leq \theta(\| x \| )$ 
for $ x \in A$ near $0 \in R^n$.
This means that $A \subset ST_{\theta}(LD(A))$ as germs at
$0 \in R^n$.
\end{proof}

\begin{prop}\label{oppositekey}
Suppose that $A$ is definable. 
Then, for any $\theta_1 \in \Phi$, there is $\theta_2 \in \Phi$ 
such that $ST_{\theta_1}(LD(A)) \subset ST_{\theta}(A)$
as germs at $0 \in R^n$, for any $\theta \in \Phi$ 
with $\theta \ge \theta_2$.
\end{prop}

\begin{proof}
Let $\theta_1 \in \Phi$.
Then, by Corollary \ref{cor1} (1), we have
$D(ST_{\theta_1}(LD(A)) = D(LD(A)) = D(A)$.
Using the same arguments for
$g(x) = {d(x,A) \over \| x \|}$ ($x \in ST_{\theta_1}(LD(A)), x \neq 0$)
as in the proof of Proposition \ref{okeyproperty},
we can find $\theta_2 \in \Phi$ such that 
$ST_{\theta_1}(LD(A)) \subset ST_{\theta_2}(A)$.
Hence $LD(A) \subset ST_{\theta_2}(A)$ as germs at $0 \in R^n$. 
\end{proof}

By Propositions \ref{okeyproperty}, \ref{oppositekey}, we have:

\begin{thm}\label{eqtheorem} 
If $A$ is definable, then $A$ is $ST$-equivalent to $LD(A)$.
\end{thm}

In a similar way to
Corollary 4.15 in \cite{koikepaunescu}, using Propositions 
\ref{okeyproperty}, \ref{oppositekey}
and Lemma \ref{osandwichlemma},
we can show the following result:

\begin{cor}\label{cor4}
Suppose that $h(A)$, $h(B)$ are definable.
If $D(h(A)) \subset D(h(B))$, then there is 
$\theta \in \Phi$ such that $A \subset ST_{\theta}(B)$
as germs at $0 \in R^n$.
\end{cor}

By Corollary \ref{cor4} and Lemma \ref{keygeneral}, we have:

\begin{thm}\label{iff}
Suppose that $h(A)$, $h(B)$ are definable.
Then the following conditions are equivalent.

(1) $D(h(A)) \subset D(h(B))$.

(2) There is $\theta \in \Phi$ such that $A \subset ST_{\theta}(B)$
as germs at $0 \in R^n$.
\end{thm}


\bigskip
\section{Sequence Selection Property}
\label{week transversality}
\medskip

In this section we discuss directional properties of sets 
with the sequence selection property, denoted by (SSP) for short,
over an Archimedean real closed field $R$.
The set $h(LD(A))$ takes a very important role in the proof
of our main theorem.
In \cite{koikepaunescu} it is  shown that over the field of
real numbers $\R$, $h(LD(A))$  satisfies condition
(SSP) provided that $A$ and $h(A)$ are both  subanalytic.
We give an improvement of this result here.

Let us recall the notion of  sequence selection 
property. 

\begin{defn}\label{SSP}
Let $A \subset R^n$ be a set-germ at $0 \in R^n$
such that $0 \in \overline{A}$.
We say that $A$ satisfies {\em condition} $(SSP)$,
if for any sequence of points $\{ a_m \}$ of $R^n$
tending to $0 \in R^n$ such that 
$\lim_{m \to \infty} {a_m \over \| a_m \| } \in D(A)$,
there is a sequence of points $\{ b_m \} \subset A$ such that
$$
\| a_m - b_m \| \ll \| a_m \|, \ \| b_m \| . 
$$ 
\end{defn}

We have some remarks on $(SSP)$ (cf. \cite{koikepaunescu}).

\begin{rem}\label{remark41} 
Condition $(SSP)$ is $C^1$ invariant,
but not bi-Lipschitz invariant.
\end{rem}

\begin{rem}\label{remark42}
Let $A \subset R^n$ be a set-germ at $0 \in R^n$
such that $0 \in \overline{A}$.

(1) The cone $LD(A)$ satisfies condition $(SSP)$.

(2) If $A \subset \R^n$ is subanalytic, then it satisfies condition $(SSP)$.

(3) If $A \subset R^n$ is definable in an o-minimal structure, 
then it satisfies condition $(SSP)$.
\end{rem}

\begin{rem}\label{remark40} 
We can describe condition $(SSP)$ without using the
convergence of a sequence of points as follows: 

\vspace{3mm}

$\forall \epsilon > 0 \in R$ $\forall \delta > 0 \in R$ 
$\forall x \ne 0$ ($\| x \| \le \delta$,
$dist({x \over \| x \|}, D(A)) \le \epsilon$,
$\exists y \in A, \| x  - y \| \le \epsilon \| x \| )$
\end{rem} 

We recall the following lemma.

\begin{lem}\label{directionrel}(Lemma 5.6 in \cite{koikepaunescu} )
Let $h : (\R^n,0) \to (\R^n,0)$ be a bi-Lipschitz homeomorphism,
and let $A \subset \R^n$ such that $0\in \overline{A}$.
Then $D(h(A)) \subset D(h(LD(A))).$
If $A$ satisfies condition $(SSP)$, the equality holds.
\end{lem}

In order to show the above lemma,  the following property was used in 
\cite{koikepaunescu} :

\vspace{3mm}

Let $\{ a_m \}$ be a sequence of points of $\R^n$ tending to
$0 \in \R^n$.
Then there is a subsequence of points $\{ a_k \}$
of $\{ a_m \}$ such that
$$
\lim_{k \to \infty} {a_k \over \| a_k \|} \in S^{n-1}
= \{ x \in \R^n \ | \ \| x \| = 1 \}.
$$
This property does not always hold on an Archimedean real closed field $R$.
In fact, Lemma \ref{directionrel} is false for the real closed
field of algebraic numbers.
Let us recall Example \ref{example21}.
Let $A = \{ a_m \}$. Then $h(A) = \{ b_m \}$.
Since $D(A) = \emptyset$ we have $D(h(LD(A))) = \emptyset$.
Nevertheless $D(h(A)) = \{ (0,1) \}$.

Over an Archimedean real closed field,
we can show the following weaker result,
which is enough to show our main theorem.

\begin{lem}\label{odirectionrel}
Let $h : (R^n,0) \to (R^n,0)$ be a bi-Lipschitz homeomorphism,
and let $A \subset R^n$ such that $0\in \overline{A}$.
Suppose that $A$ is definable.
Then we have $LD(h(A)) = LD(h(LD(A))).$
\end{lem}

\begin{proof}
We first show the inclusion $LD(h(A)) \subset LD(h(LD(A))).$
By Proposition \ref{okeyproperty}. there is $\theta_1 \in \Phi$
such that $A \subset ST_{{\theta}_1}(LD(A))$ as set-germs at $0 \in R^n$.
Then, by the sandwich lemma, we have
$$
h(A) \subset h(ST_{{\theta}_1}(LD(A))) \subset ST_{{\theta}_2}(h(LD(A)))
$$
for some $\theta_2 \in \Phi$.
Therefore it follows from Corollary \ref{cor1} (2) that
$$
LD(h(A)) \subset LD(ST_{{\theta}_2}(h(LD(A)))) = LD(h(LD(A))).
$$

The opposite inclusion $LD(h(A)) \supset LD(h(LD(A)))$
follows from a similar argument,
replacing Proposition \ref{oppositekey}
with Proposition \ref{okeyproperty}.
\end{proof}

As a corollary of Lemma \ref{odirectionrel} and Remark \ref{remark30}
we have the following lemma.

\begin{lem}\label{LD(h(LD(A)))}
Let $h : (R^n,0) \to (R^n,0)$ be a bi-Lipschitz homeomorphism,
and let $A \subset R^n$ such that $0\in \overline{A}$.
If $A$ and $h(A)$ are definable, 
then $LD(h(LD(A))) = LD(h(A))$ is definable.
\end{lem}

We give one more example having condition $(SSP)$.
Using a similar argument
to Proposition 6.4 in \cite{koikepaunescu}, we can show the following:

\begin{prop}\label{h(LD(A))}
Let $h : (R^n,0) \to (R^n,0)$ be a bi-Lipschitz homeomorphism,
and let $A$, $h(A) \subset R^n$ be 
definable set-germs at $0 \in R^n$ such that
$0\in \overline{A}$.
Then the set $h(LD(A))$ satisfies condition $(SSP)$.
\end{prop}

Let us discuss more on the sequence selection property
over the field of real numbers $\R$.
In this note we consider also the notion of weak
sequence selection property, denoted by $(WSSP)$ for short.

\begin{defn}\label{WSSP}
Let $A \subset \R^n$ be a set-germ at $0 \in \R^n$
such that $0 \in \overline{A}$.
We say that $A$ satisfies {\em condition} $(WSSP)$,
if for any sequence of points $\{ a_m \}$ of $\R^n$
tending to $0 \in \R^n$ such that 
$\lim_{m \to \infty} {a_m \over \| a_m \| } \in D(A)$,
there exist a subsequence $\{ m_j \}$  and
$\{ b_{m_j} \} \subset A$ such that
$$
\| a_{m_j} - b_{m_j} \| \ll \| a_{m_j} \|, \ \| b_{m_j} \| . 
$$ 
\end{defn}

We have the following characterisation of condition $(SSP)$.

\begin{lem}\label{SSPWSSP}
Let $A \subset \R^n$ be a set-germ at $0 \in \R^n$
such that $0 \in \overline{A}$.
If $A$ satisfies condition $(WSSP)$, then it satisfies condition $(SSP)$.
Namely, conditions $(SSP)$ and $(WSSP)$ are equivalent.
\end{lem}

\begin{proof} We show that $(WSSP)$ implies $(SSP)$.
Assume that $A$ does not satisfy condition $(SSP)$.
Then there is a sequence of points $\{ a_m \}$
tending to $0 \in \R^n$ with
$\lim_{m \to \infty} {a_m \over \| a_m \|} \in D(A)$
such that for any sequence of points $\{ b_m \} \subset A$,
the following is not satisfied:
$$
\| a_m - b_m \| \ll \| a_m \| . 
$$ 
Therefore there is a subsequence of points $\{ a_{m_j} \}$ of
$\{ a_m \}$ such that
$\lim_{m_j \to \infty} {d(a_{m_j},A) \over \| a_{m_j} \|} = \alpha > 0$,
where $d(a_{m_j},A)$ denotes the distance between $a_{m_j}$ and $A$.
Taking this $\{ a_{m_j} \}$ as the first $\{ a_m \}$,
we can assume from the beginning that
$\lim_{m \to \infty} {d(a_m,A) \over \| a_m \|} = \alpha > 0$.
This implies that there does not exist a sequence of points
$\{ b_{m_j} \} \subset A$ such that
$\| a_{m_j} - b_{m_j} \| \ll \| a_{m_j} \|$.
Therefore $A$ does not satisfy condition $(WSSP)$.
\end{proof}

Using Lemmas \ref{directionrel} and \ref{SSPWSSP} we can improve 
Proposition 6.4 in \cite{koikepaunescu} as follows:

\begin{thm}\label{2SSP}
Let $h : (\R^n,0) \to (\R^n,0)$ be a bi-Lipschitz homeomorphism,
and let $A \subset \R^n$ such that $0\in \overline{A}$.
Assume that $A$ satisfies condition $(SSP)$.
Then $h(A)$ satisfies condition $(SSP)$, if and only if,
$h(LD(A))$ satisfies condition $(SSP)$.
\end{thm}

\begin{proof}
We first show the ``only if" part.
By assumption, $A$ satisfies condition $(SSP)$.
Therefore it follows from Lemma \ref{directionrel} that
$D(h(LD(A))) = D(h(A))$.
Let $\{ y_m \}$ be an arbitrary sequence of points of $\R^n$
tending to $0 \in \R^n$ such that
$$
\lim_{m \to \infty} {y_m \over \| y_m \|} \in D(h(LD(A))) = D(h(A)).
$$
Let $y_m = h(x_m)$ for each $m$.
Since $h(A)$ satisfies condition $(SSP)$, there is a sequence of points
$\{ z_m \} \subset A$ such that
$$
\| h(x_m) - h(z_m) \| \ll \| h(x_m) \| , \ \| h(z_m) \| .
$$

On the other hand, there is a subsequence $\{ z_{m_j} \}$
of $\{ z_m \}$ such that 
$\lim_{{m_j} \to \infty} {z_{m_j} \over \| z_{m_j} \|} \in D(A)$.
Since $LD(A)$ satisfies condition $(SSP)$, there is a sequence of points
$\{ \theta_{m_j} \} \subset LD(A)$ such that
$$
\| z_{m_j} - \theta_{m_j} \| \ll \| z_{m_j} \| , \ \| \theta_{m_j} \| .
$$
It follows from the bi-Lipschitz of $h$ that
$$
\| h(z_{m_j}) - h(\theta_{m_j}) \| \ll \| h(z_{m_j}) \| , \ 
\| h(\theta_{m_j}) \| .
$$
Then we have
$$
\| h(x_{m_j}) - h(\theta_{m_j}) \| \le
\| h(x_{m_j}) - h(z_{m_j}) \| + \| h(z_{m_j}) - h(\theta_{m_j}) \| 
\ll \| h(z_{m_j}) \| .
$$
Therefore we have
$$
\| h(x_{m_j}) - h(\theta_{m_j}) \| \ll \| h(x_{m_j}) \| , \ 
\| h(\theta_{m_j}) \| .
$$
Thus $h(LD(A))$ satisfies condition $(WSSP)$, and also condition
$(SSP)$ by Lemma \ref{SSPWSSP}.

The ``if" part can be proved in a similar way.
\end{proof}


\bigskip
\section{Proof of Main Theorem}
\label{MTproof}
\medskip

Our main theorem is proved in the same way as
Theorem \ref{maintheorem}.
Since the reduction arguments in \S 6 of \cite{koikepaunescu}
work also for definable sets, it only suffices  to show
the following proposition.

\begin{prop}\label{reductions}
Let $h : (R^n,0) \to (R^n,0)$ be a bi-Lipschitz homeomorphism,
and let $A \subset R^n$ such that $0\in \overline{A}$.
Suppose that $A$, $h(A)$ are definable. Then
$$
\dim h(LD(A)) \ge \dim LD(h(LD(A))).
$$ 
\end{prop}

As mentioned in the introduction, Theorem \ref{maintheorem}
was proved using sea-tangle properties,
sequence selection properties and volume arguments.
In this section we discuss volume arguments and give
a proof of our main theorem.
In order to avoid considering the volume on a general
Archimedean real closed field $R$, we take the closure
of a subset of $R^n$ in $\R^n$.
Note that $R^n$ is dense in $\R^n$.

For a subset $A$ of $R^n$ ($\subset \R^n$),
let $\overline{A}^{\R}$ denote the closure of $A$ in $\R^n$
(not in $R^n$), and let $B_{\epsilon}(0)$ denote a closed
$\epsilon$ ball in $\R^n$ centred at $0 \in \R^n$
for $\epsilon > 0$, $\epsilon \in \R$.

Let $f, \ g : [0,\delta ) \to \R$, $\delta > 0$, be non-negative
functions, where $[0,\delta )$ is a half open interval of $\R$. 
If there are real numbers $K > 0$, $0 < \delta_1 \le \delta$ 
such that
$$
f(\epsilon ) \le K g(\epsilon ) \ \ \text{for} \ \ 0 \le \epsilon \le \delta_1,
$$
then we write $f \precsim g$ (or $g \succsim f$).
If $f \precsim g$ and $f \succsim g$, we write $f \thickapprox g$.

We can easily see the following property on volumes.

\begin{lem}\label{ctimes} Let $A \subset R^n$ such that
$0 \in \overline{A}$.
Then 
$$
Vol(\overline{ST_{c\theta}(A)}^{\R} \cap B_{\varepsilon}(0)) 
\thickapprox Vol(\overline{ST_{\theta}(A)}^{\R} \cap
B_{\varepsilon}(0))
$$ 
for $\theta \in \Phi$ and $c > 0$. 
\end{lem}

Using a similar argument as in 
Lemma 7.1 of \cite{koikepaunescu}, we can show the following lemma.

\begin{lem}\label{volumelemma} Let $\alpha$, $\beta$ 
be linear subspaces of $R^n$.
Suppose that $\dim \alpha < \dim \beta$.
Then, for $\theta \in \Phi$,
$$
\lim_{\epsilon \to 0} { Vol(\overline{ST_{\theta}(\alpha )}^{\R}
\cap B_{\epsilon}(0)) \over Vol(\overline{ST_{\theta}(\beta )}^{\R}
\cap B_{\epsilon}(0)) } = 0.
$$
\end{lem}

We have the following volume properties analogous to those 
in \cite{koikepaunescu}.

\begin{prop}\label{conevolume}
Let $\alpha$, $\beta \subset R^n$ be definable cones at $0 \in R^n$.
Suppose that $\dim \alpha < \dim \beta$.
Then, for $\theta \in \Phi$,
$$
\lim_{\epsilon \to 0} { Vol(\overline{ST_{\theta}(\alpha )}^{\R}
\cap B_{\epsilon}(0))
\over Vol(\overline{ST_{\theta}(\beta )}^{\R} \cap B_{\epsilon}(0)) } = 0.
$$
\end{prop}

\begin{proof}
Let $\gamma$ be a definable cone at $0 \in R^n$ of dimension $r$,
and let $M$ be an $r$-dimensional linear subspace of $R^n$.
Then the proposition follows easily from Lemma \ref{volumelemma}
and the fact that
\begin{equation}\label{equiv01}
Vol(\overline{ST_{\theta}(\gamma)}^{\R} \cap B_{\epsilon}(0)) 
\thickapprox 
Vol(\overline{ST_{\theta}(M)}^{\R} \cap B_{\epsilon}(0))
\end{equation}
for $\theta \in \Phi$.

Let us show (\ref{equiv01}).
We may assume that $\gamma$ is equidimensional.
Then there exist a finite partition of $\gamma$ into
$r$-dimensional definable cones $\gamma_1, \cdots , \gamma_s$
with $0 \in R^n$ as a vertex, and $r$-dimensional linear subspaces 
$M_1, \cdots , M_s$ of $R^n$ such that for each orthogonal 
projection $\Pi_i : R^n \to M_i$, $1 \le i \le s$,
$\gamma_i$ is expressed as the graph of a definable map
from $\Pi_i(\gamma_i) \subset M_i$ and
the diameter of $ST_{\theta}(\gamma ) \cap \Pi_i^{-1}(u)$,
$u \in \Pi_i(\gamma_i)$, is less than or equal to $4\theta$.
Then we can see that
\begin{equation}\label{equiv02}
Vol(\overline{ST_{\theta}(\gamma)}^{\R} \cap B_{\epsilon}(0)) \precsim
Vol(\overline{ST_{4\theta}(M)}^{\R} \cap B_{\epsilon}(0)).
\end{equation}
On the other hand, we may assume that one of $\Pi_i(\gamma_i)$'s is
a closed cone in $M_i$, taking a finite subdivision of $\gamma_i$'s 
if necessary.
Then we can see that $M_i$ is covered with a finite number of 
$r$-dimensional closed cones of the same size as $\Pi_i(\gamma_i)$.
It follows that
$$
Vol(\overline{ST_{\theta}(M_i)}^{\R} \cap B_{\epsilon}(0)) \precsim
Vol(\overline{ST_{\theta}(M_i) \cap \Pi_i^{-1}(\Pi_i(\gamma_i))}^{\R}
\cap B_{\epsilon}(0)) \precsim
Vol(\overline{ST_{\theta}(\gamma_i)}^{\R} \cap B_{\epsilon}(0)).
$$
Therefore we have
\begin{equation}\label{equiv03}
Vol(\overline{ST_{\theta}(\gamma)}^{\R} \cap B_{\epsilon}(0)) \succsim
Vol(\overline{ST_{\theta}(M)}^{\R} \cap B_{\epsilon}(0)).
\end{equation}
Then (\ref{equiv01}) follows from (\ref{equiv02}), (\ref{equiv03})
and Lemma \ref{ctimes}.
\end{proof}

The next lemma follows in the same way as Lemma 7.3 
in \cite{koikepaunescu}:

\begin{lem}\label{definabledimension}
Let $A \subset R^n$ be a definable
 set-germ at $0 \in R^n$
such that $0 \in \overline{A}$.
Then we have $\dim LD(A) \le \dim A$.
\end{lem}

\begin{prop}\label{STvolume} 
Let $A, B$ be set-germs at $0$ in $R^n$
such that $0 \in\ \overline{A} \cap \overline{B}$. 
Suppose that $A$ and $B$ are $ST$-equivalent.
Then there is $\theta_1 \in \Phi$ such that
$$
Vol(\overline{ST_{\theta}(A)}^{\R} \cap B_{\epsilon}(0)) \thickapprox
Vol(\overline{ST_{\theta}(B)}^{\R} \cap B_{\epsilon}(0))
$$
for any $\theta \in \Phi$ with $\theta \ge \theta_1$.
\end{prop}

\begin{proof}
Let $\theta_3, \ \theta_4 \in \Phi$ such that $A \subset
ST_{{\theta_3 \over 2}}(B), \ B \subset ST_{{\theta_4 \over 2}}(A)$. 
Take $\theta_1 \in \Phi$ so that
$\theta_1(t) \geq 2 \max(\theta_3(2t),\theta_4(2t))$ for $0<t<1$. 

\vspace{3mm}

\noindent {\em Claim}. For any $\theta\in\Phi$ with $\theta\geq\theta_1$,
we have
$$
ST_{{\theta \over 2}}(A) \subset ST_{{\theta \over 2}}
(ST_{{\theta_3 \over 2}}(B)) \subset ST_{2\theta}(B) 
\subset ST_{4\theta}(A)
$$
as germs at $0 \in R^n$.

\begin{proof}
To see the second inclusion, 
let $x \in ST_{{\theta \over 2}}(ST_{{\theta_3 \over 2}}(B))$ 
with $\theta(\|x\|) \le 1$. 
Then there exists $y \in R^n$ such that
$d(y,B) \leq {\theta_3 \over 2}(\|y\|)\| y \|$ and $d(x,y) \leq
\theta(\|x\|)\|x\|$.
We also have 
$\|y\| \leq \| x \| + \theta(\| x \| )\| x \| \leq 2 \| x \|$. 
Take $z \in B$ such that $d(y,z) \leq \theta_3(\| y \|)\| y \|$. 
Then
\[\begin{array}{llll}
    d(x,z) &\leq d(x,y)+d(y,z) \\
    &\leq \theta(\|x\|)\|x\|+\theta_3(\|y\|)\|y\|\\
    &\leq \theta(\|x\|)\|x\|+\theta_3(2\|x\|)2\|x\| \\
    &\leq \theta(\|x\|)\|x\|+\theta(\|x\|)\|x\| = 2\theta(\|x\|)\|x\|.
  \end{array}
\]
This implies $x \in ST_{2\theta}(B)$, and hence
$ST_{{\theta \over 2}}(ST_{{\theta_3 \over 2}}(B)) \subset ST_{2\theta}(B)$.

Let $x \in ST_{2\theta}(ST_{{\theta_4 \over 2}}(A))$ 
with $\theta(\|x\|) \le {1 \over 3}$. 
Similarly as above, we can show
$$
ST_{2\theta}(B) \subset ST_{2\theta}(ST_{{\theta_4 \over 2}}(A)) 
\subset ST_{4\theta}(A).
$$ 
\end{proof}

The statement of the proposition follows from
this claim and Lemma \ref{ctimes}.
\end{proof}

The following corollary is an obvious consequence of
Theorem \ref{eqtheorem}, Lemma \ref{definabledimension} and
Propositions \ref{conevolume} and  \ref{STvolume}.

\begin{cor}\label{volumeratio}
Let $\alpha \subset R^n$ be a definable set-germ at $0 \in R^n$ 
such that $0 \in \overline{\alpha}$,
and let $\beta \subset R^n$ be a definable cone at $0 \in R^n$.
Suppose that $\dim \alpha < \dim \beta$.
Then there is $\theta_1 \in \Phi$ such that
$$
\lim_{\epsilon \to 0} { Vol(\overline{ST_{\theta}(\alpha )}^{\R}
\cap B_{\epsilon}(0))
\over Vol(\overline{ST_{\theta}(\beta )}^{\R} \cap B_{\epsilon}(0)) } = 0
$$
for any $\theta \in \Phi$ with $\theta \ge \theta_1$.
\end{cor}

We next describe a key lemma needed for proving our main theorem.

\begin{lem}\label{maintool}
Let $h : (R^n,0) \to (R^n,0)$ be a bi-Lipschitz homeomorphism,
let $E \subset R^n$ be a definable set-germ at $0 \in R^n$
such that $0 \in \overline{E}$, and let $F := h(E)$.
Suppose that $F$ and $LD(F)$ are $ST$-equivalent and
$LD(F)$ is definable.
Then we have $\dim LD(F) \le \dim E$.
\end{lem}

\begin{proof}
Since $h : (R^n,0) \to (R^n,0)$ is a bi-Lipschitz homeomorphism
and $R^n$ is dense in $\R^n$, $h$ has a natural extension to
a bi-Lipschiz homeomorphism $\overline{h} : (\R^n,0) \to (\R^n,0)$.
Then we have
$$
\overline{h(ST_{\theta}(E))}^{\R} =
\overline{h}(\overline{ST_{\theta}(E)}^{\R}) \ \
\text{for} \ \theta \in \Phi.
$$
Therefore it follows from Lemmas \ref{osandwichlemma} and \ref{ctimes} that
\begin{equation}\label{equiv1}
Vol(\overline{ST_{\theta}(F)}^{\R} \cap B_{\epsilon}(0)) \thickapprox  
Vol(\overline{ST_{\theta}(E)}^{\R} \cap B_{\epsilon}(0)).
\end{equation}

On the other hand, $F$ and $LD(F)$ are $ST$-equivalent.
By Proposition \ref{STvolume},
there is $\theta_1 \in \Phi$ such that
\begin{equation}\label{equiv2}
Vol(\overline{ST_{\theta}(F)}^{\R} \cap B_{\epsilon}(0)) \thickapprox
Vol(\overline{ST_{\theta}(LD(F))}^{\R} \cap B_{\epsilon}(0))
\end{equation}
for any $\theta \in \Phi$ with $\theta \ge \theta_1$.

By (\ref{equiv1}) and (\ref{equiv2}), we have
\begin{equation}\label{equiv3}
1 \thickapprox 
{Vol(\overline{ST_{\theta}(F)}^{\R} \cap B_{\epsilon}(0)) \over
Vol(\overline{ST_{\theta}(LD(F))}^{\R} \cap B_{\epsilon}(0))} \thickapprox  
{Vol(\overline{ST_{\theta}(E)}^{\R} \cap B_{\epsilon}(0)) \over
Vol(\overline{ST_{\theta}(LD(F))}^{\R} \cap B_{\epsilon}(0))}
\end{equation}
for any $\theta \in \Phi$ with $\theta \ge \theta_1$.
Assume that $\dim LD(F) > \dim E$.
Then, by Corollary \ref{volumeratio}, 
the RHS ratio in (\ref{equiv3}) tends to $0$ as 
$\epsilon \to 0$, if $\theta$ is sufficiently big.
This is a contradiction.
Thus we have $\dim LD(F) \le \dim E$.
\end{proof}

Let us show Proposition \ref{reductions}.
The sets $A$ and $ h(A)$ are assumed definable.
Therefore, by Lemma \ref{LD(h(LD(A)))}, $LD(h(A))=LD(h(LD(A)))$
is definable, and it follows from Theorem \ref{eqtheorem}
that $LD(A)$ is $ST$-equivalent to $A$.
Then, by Proposition \ref{ST-equivalence},
$h(LD(A))$ is $ST$-equivalent to $h(A)$.
In addition, it follows from the definability of $h(A)$, that
$h(A)$ is $ST$-equivalent to $LD(h(A)) = LD((h(LD(A)))$.
Since the $ST$-equivalence is an equivalence relation,
$h(LD(A))$ is $ST$-equivalent to $LD(h(LD(A)))$.
Thus the proposition follows from Lemma \ref{maintool} 
with $E = LD(A)$ and $F = h(LD(A))$,
since $\dim h(LD(A)) = \dim LD(A)$.

This completes the proof of our main theorem.


\bigskip
\section{General real closed field case}
\label{generalcase}
\medskip

In this section we formulate and prove our main theorem for an arbitrary
 real
 closed field.
Let $R$ denote a real closed field with an o-minimal structure 
and consider the topology on $R$ given by the open intervals of $R$,
analogous to that on  $\R$. 
We have already proved our main theorem for an arbitrary 
Archimedean real closed
field. 
An example of a non-Archimedean real closed field is the field of 
Puiseux series, where a {\it Puiseux series} is a power series of the form 
$\sum_{i=k}^\infty a_i t^{i/p}$ for $a_i \in \R$,
$p > 0 \in \N$ and $k \in \Z$ 
such that $\sum_{i=\max (0,k)}^\infty a_i t^i$ is a formal
(convergent) power series in one variable $t$. 
One reason for why we consider problems on a general real closed field $R$, 
is that some problems on $\R$ are solved by replacing $\R$ with $R$. 
Actually, the Hilbert 17th problem is a famous illustration of this.

In order to treat our main theorem for $R$,
we need to modify the previous definitions.
Let $A$ and $B$ always denote subsets of $R^n$.
Set
$$
dist(A,B) = \{ t \ge 0 \in R \ | \ \forall a \in A \ \forall b \in B, 
t \le \| a - b \| \}.
$$
Here for $(x_1, \cdots , x_n) \in R^n$ we define $\| (x_1, \cdots , x_n)\|$ 
to be $\max_{i=1, \cdots ,n} |x_i|$.
Then $dist(A,B)$ is a closed connected subset of $R$. 
In case $R = \R$, it is basically the closed interval with ends 0 and the 
usual 
distance between $A$ and $B$. 
In the general case if $A$ and $B$ are definable, $dist(A,B)$ is also a 
closed interval with ends 
0 and 
a number in $R$. 
(By abuse of notation $dist(A,B)$ denotes also the right end number 
in the definable case.)
However, for general subsets this is not the case, even if $R$ 
is Archimedean. 
Nevertheless in the Archimedean case we may 
understand $dist(A,B)$ as a real number. 
For a general $R$ we do not know such an extension field, and
therefore we have to define $dist(A,B)$ to be a subset of $R$.
For subsets $X$ and $Y$ of $R$, $X \le Y$ means by definition that 
$\forall x \in X\ \exists y\in Y\ x\le y$, and $X + Y$ denotes 
the set $\{ x + y \ | \ x \in X,y \in Y \}$.

Let $D(A)$ denote the subset of 
$S^{n-1} = \{ x \in R^n \ | \ \| x \| = 1 \}$ 
consisting of points $a$ such that
$\forall \epsilon > 0 \in R \ \forall \delta > 0 \in R\ 
\exists x\in A-\{0\}$ with $\| x \| \le \delta,\ 
dist(a,{x \over \| x \|})< \epsilon.
$
In the Archimedean case, $D(A)$ coincides with that in Definition 2.2. 
It may be empty for general $A$, but if $A$ is definable 
then it is definable and not empty. 
We define $LD(A)$ in the same way as before, it is definable 
if $A$ is definable. 

In the sense of the above notions of distance and direction set,
we have the following:

\begin{thm}\label{non-Archimedean}
Let $A$, $B \subset R^n$ be definable set-germs at $0 \in R^n$ 
such that $0 \in \overline{A} \cap \overline{B}$, and let 
$h : (R^n,0) \to (R^n,0)$ be a bi-Lipschitz homeomorphism.
Suppose that $h(A), \ h(B)$ are also definable.
Then we have
$$
\dim (D(h(A)) \cap D(h(B))) = \dim (D(A) \cap D(B)).
$$
\end{thm}

Let us show Theorem \ref{non-Archimedean}.
Define $\Phi$ as in the Archimedean case. 
Let $A \subset R^n$ such that $0 \in \overline{A}$.
For $\theta \in \Phi$, we define the {\em sea-tangle neighbourhood}
of $A$ by
$$
ST_{\theta}(A) := \{ x \in  R^n \ | \ dist(x,A) \le 
\theta (\| x \|)\| x \| \}.
$$
If $A$ is definable, so is $ST_{\theta}(A)$. 
If $R$ is Archimedean, $ST_{\theta}(A)$
coincides with that in Definition \ref{oseatanglenbd}.
$ST$-{\em equivalence} is also defined as before. 

Using the notions of $\Phi$, sea-tangle neighbourhood and
condition $(SSP)$ as above,
Lemma \ref{osandwichlemma}, Proposition \ref{ST-equivalence}, 
Lemma \ref{keygeneral}, Corollary \ref{cor1}, 
Proposition \ref{okeyproperty}, Proposition \ref{oppositekey}, 
Theorem \ref{eqtheorem}, Corollary \ref{cor4}, 
Theorem \ref{iff}, Lemma \ref{odirectionrel}, 
Lemma \ref{LD(h(LD(A)))} and Proposition \ref{h(LD(A))}
are all proved for $R$, in the same way as before 
by replacing sequences with filters. 
There are no essential modifications until \S 5, because those 
arguments work in the family of definable sets and their images 
under bi-Lipschitz homeomorphisms. 
However we need to modify \S 6, in particular,
the volume arguments.

For a subset $A \subset R^n$ we define $\dim A$
to be the dimension of the topological space $A$  as in the dimension 
theory (see \cite{hurewiczwallman}). 
There are no problems with this definition because $\dim A = \dim h(A)$ 
for a homeomorphism $h$ of $R^n$, if $A \subset B$ then $\dim A \le \dim B$,
and if $A$ is definable then $\dim A$ coincides with the largest integer 
$k$, such that there exists a definable imbedding of $R^k$ into $A$.
Then Theorem \ref{non-Archimedean} follows from Proposition \ref{reductions}
for $R$ as in the Archimedean case. 
We will refer to Proposition \ref{reductions} for $R$ as to the
{\em Generalised Proposition} \ref{reductions}. 
If the germ of $A$ at $0 \in R^n$ in the statement of the 
Generalised Proposition 
\ref{reductions} is of dimension $n$, then $\dim h(LD(A)) = \dim LD(A) = n$ 
and the Generalised Proposition \ref{reductions} holds true. 
Therefore we consider only subsets of 
$\{(x_1, \cdots ,x_n) = (x',x_n) \in R^n \ | \ \| x' \| \le c x_n,
x_n \ge 0 \}$, $c > 0 \in R$, of dimension less than $n$. 
For each $\epsilon \ge 0 \in R$, set 
$$
A_{\epsilon} := \{ x' \in R^{n-1} \ | \ (x',\epsilon) \in A \}. 
$$
Then we can change the previous definition of $ST_\theta(A)$ to 
$$
\{ (x',x_n) \in R^n \ | \ x_n \ge 0, dist(x',A_{x_n}) \le \theta(x_n) x_n \}, 
$$
which is easier to calculate.

The problem is how to define the volume of $A$. 
We do not have any good definition of volume in the general case. 
However the following definition is sufficient for our purpose. 
Assume that $A$ is bounded and definable.
 
In the following we will define $Vol A$. 
For the simplest case of an open cube 
$O = (a_1,b_1) \times \cdots \times (a_n,b_n)$ in $R^n$ with
$a_i \le b_i \in R$, let $Vol O$ be the expected product $(b_1-a_1) \cdots (b_n-a_n)$. 
Next consider the case where $A$ is the following set:
$$
\{ (x_1,...,x_n) \in R^n \ | \ a_1\!<\!x_1\!<\!b_1,
a_2\!<\!x_2\!<\!\phi_2(x_1),...,a_n\!<\!x_n\!<\!\phi_n(x_1,...,x_{n-1}) \}
$$
for $a_1, \cdots , a_n, b_1 \in R$ and bounded definable $C^0$ 
functions $\phi_2$ on $(a_1,\,b_1), \cdots , \phi_n$ on 
\begin{eqnarray*}
& &  \{ (x_1, \cdots ,x_{n-1}) \in R^{n-1} \ |\\
\notag
& & a_1 <x_1 < b_1, a_2 < x_2 < \phi_2(x_1), \cdots , a_{n-1} < x_{n-1}
< \phi_{n-1}(x_1,\cdots ,x_{n-2}) \}. 
\end{eqnarray*}

\noindent We  write $A = A_{a_1,...,a_n, b_1,\phi_2,...,\phi_n}$. 
For each $k=1,2, \cdots$, let $Vol_k A$ denote the maximal number 
of $\sum_i Vol O_i$ where $\{O_i\}$ runs over all $k$ 
disjoint open cubes contained 
in $A$ (they exist since $A$ is definable). 
Note that $\{ Vol_k A \}_{k=1,2,...}$ is increasing. 
From now on we identify $Vol_k A$ with 
$\{ t \ge 0 \in R \ | \ t \le Vol_k A \}$. 
Set $Vol A := \cup_k Vol_k A$, which is a convex subset 
of $\{ t \ge 0 \in R \}$ and contains 0. 

Consider the set $A$ of the form 
\begin{eqnarray*}
& &  \{(x_1,...,x_n) \in R^n \ |\\ 
\notag
& & 
a_1\!<\!x_1<b_1,\psi_2(x_1)\!<\!x_2\!<\!\phi_2(x_1),...,
\psi_n(x_1,...,x_{n-1})\!<\!x_n\!<\!\phi_n(x_1,...,x_{n-1})\}
\end{eqnarray*}

\noindent for $a_1,b_1\in R$ and bounded definable $C^0$ functions 
$\psi_i$ and $\phi_i$ on 
$$
\{(x_1,...,x_{i-1})\!\in\! R^{i-1} \ | \ a_1\!<\!x_1\!<\!b_1,...,
\psi_{i-1}(x_1,...,x_{i-2})\!<\!x_{i-1}\!<\!\phi_{i-1}(x_1,...,x_{i-2})\}
$$
with $\psi_i < \phi_i, \ i=2, \cdots ,n$. 
We  call $A$ {\em of cell form} and write 
$A\!=\!A_{a_1,\psi_2,...,b_1,\phi_2,...}$.
If we define $Vol A$ for $A$ of cell form in the same way as above, 
then our arguments do not work. 
For example, let $\epsilon>0\in R$ be smaller than any positive real 
number and set $A = A_{0,x,1,\epsilon+x}$ in $R^2$. 
Then 
$$
Vol A = \{ a \ge 0 \in R \ | \ \exists k \in \N \ a \le k \epsilon^2\}. 
$$
However we expect $Vol A = [0,\,\epsilon]$, both  from the context of 
the proofs in \S6 and from the following arguments. 
Here we introduce an artificial different definition of volume. 
Choose $a_2, \cdots ,a_n \in R$ so small that 
$a_2 < \psi_2, \cdots ,a_n < \psi_n$. 
Then 
$$
A = A_{a_1,...,a_n,b_1,\phi_2,...,\phi_n} - 
\cup \overline{A_{a_1,...,a_n,b_1,\rho_2,...,\rho_n}},
$$
where $\{\rho_2,...,\rho_n\}$ satisfy $\rho_i = \psi_i$ 
for one $i$ and $\rho_j = \phi_j$ for the other $j's$. For 
such distinct families $\{ \rho_i \}$ and $\{ \rho'_i \}$  we have
$$
A_{a_1,...,a_n,b_1,\rho_2,...,\rho_n} \cap A_{a_1,...,a_n,b_1,\rho'_2,
...,\rho'_n} 
= A_{a_1,...,a_n,b_1,\min(\rho_2,\rho'_2),...,\min(\rho_n,\rho'_n)}.
$$
If $R = \R$ then 
$$
Vol A = \sum p_{\xi_2,...,\xi_n} Vol A_{a_1,...,a_n,b_1,\xi_2,...,\xi_n}
$$
in the usual sense of volume for some integers $p_{\xi_2,...,\xi_n}$
where $\{ \xi_2, \cdots ,\xi_n \}$ are so that for all $i$,  $\xi_i = \psi_i$ or 
$\xi_i = \phi_i$. 
Here the $p_{\xi_2,...,\xi_n}$'s do not depend on the special choice 
of $\psi_i$ and $\phi_i$. 
For a general $R$, set 
\begin{eqnarray*}
& & Vol_k A := \sum p_{\xi_2,...,\xi_n} 
Vol_k A_{a_1,...,a_n,b_1,\xi_2,...,\xi_n}, \\
& & Vol A := \cap_{l=1}^\infty \cup_{k=l}^\infty Vol_k A,
\end{eqnarray*}

\noindent i.e., $\epsilon \ge 0 \in Vol A$ if and only if 
$\forall l \in \N \ \exists k \ge l \in \N \ \epsilon \in Vol_k A$. 

Note that $Vol A_{0,x,1,\epsilon+x} = [0,\,\epsilon]$,
and if $A$ is of cell form and is a disjoint union of a definable set 
of dimension less than $ n$ and a finite number of definable sets $A_i$ 
of cell form, then $Vol A = \sum_i Vol A_i$.

Let $A$ be a general bounded definable set in $R^n$. 
Then, by the Cell Decomposition Theorem, $A$ is a disjoint union 
of a definable set of dimension less than $ n$ and a finite number 
of definable sets $A_i$ of cell form. 
Set $Vol A = \sum_i Vol A_i$. 
Then $Vol A$ does not depend on the choice of cell decomposition. 
If $B$ is open, bounded and definable and contains $\overline A$,
then $Vol A\varsubsetneq Vol B$; 
$Vol A=Vol h(A)$ for the map 
$h : R^n \ni (x_1,...,x_n) \to (x_{\tau(1)},...,x_{\tau(n)}) \in R^n$ 
where $\tau$ is a permutation of $\{1,...,n\}$. 
For disjoint bounded definable sets $A_1$ and 
$A_2$ we have  $Vol(A_1 \cup A_2) = Vol A_1 + Vol A_2$.

This is the definition of volume of definable subsets of $R^n$. 

For a definable subset $A$ of $R^n$ and for each $\epsilon \ge 0 \in R$, 
$Vol A_\epsilon$ is a subset of $R$. 
Note that $Vol A_\epsilon$ is calculated regarding  $A_\epsilon$ 
as a subset of $R^{n-1}$. 
We are interested in the correspondence 
$[0,\,\infty) \ni \epsilon \to Vol A_\epsilon \subset R$. 
Let $Vol_A$ denote the correspondence. 
Let $f$ and $g$ be maps from $[0,\,\infty)$ to the power set $\frak B(R)$. 
If there are $K > 0$ and $\delta > 0$ in $R$ such that 
$$
f(\epsilon) \subset Kg(\epsilon) \subset K^2f(\epsilon) \quad
\text{for} \quad \epsilon \in [0,\,\delta],
$$
we write $f \approx g$. 
By $\lim_{\epsilon \to 0} \frac{g(\epsilon)}{f(\epsilon)} = 0$,
we mean that $\forall \epsilon > 0 \in R\ \exists \delta > 0 \in R 
\ \forall \epsilon_1 \in (0,\,\delta] \ 
g(\epsilon_1) \subset \epsilon f(\epsilon_1) \not= \{ 0 \}$.

For subsets $X$ and $Y$ of $R^n$, $Vol_X \approx Vol_Y$ means that 
there are definable sets $A_1,A_2,B_1$ and $B_2$ in $R^n$ such that 
$$
A_1 \subset X \subset A_2, \ B_1 \subset Y \subset B_2\ \text{and}
\ Vol_{A_1} \approx Vol_{A_2} \approx Vol_{B_1} \approx Vol_{B_2}.
$$
In the same way we define 
$\lim_{\epsilon\to0}\frac{Vol_X(\epsilon)}{Vol_Y(\epsilon)}=0$.

Because of  Lemma \ref{ctimes}, we may expect 
$Vol_{ST_{c\theta}(X)} \approx Vol_{ST_\theta(X)}$, for 
$X \subset R^n, \ \theta \in \Phi$ and $c > 0 \in R$. 
However we do not know whether this is the case. 
We first prove:
\begin{lem}\label{gctimes}
Let $A$ be a definable cone at $0 \in R^n$. 
Then
$$
Vol_{ST_{c\theta}(A)} \approx Vol_{ST_\theta(A)}
$$
for $\theta \in \Phi$ and $c > 0 \in R$. 
\end{lem}

\begin{proof}
Assume $c > 1$, and let $A$ be the cone with base 
$C \times \{1\} \subset R^{n-1} \times R$, where we 
 assume that $C$ is closed. 
Since $C$ admits a finite stratification with definable $C^1$ manifolds, 
we only need to prove the lemma for the cones with vertex $0$
and base each of the strata. 
In addition, as in \S II.1 in \cite{shiota}, we can choose the 
stratification so that for each stratum $C_1$ of dimension $k$, 
there exist $n_1 < \cdots <n_k$ in $\{1, \cdots ,n\}$ (assume, for the simplicity of notation, that $n_i=i, \, i=1, \cdots ,k$),  
 such that 
the restriction to $\overline{C_1}$ (not only to $C_1$)
of the projection 
$p : R^{n-1} \ni (x_1, \cdots ,x_{n-1}) \to (x_1, \cdots ,x_k) \in R^k$ 
is injective, i.e., $\overline{C_1}$ is the graph of some $C^0$ map 
$\alpha = (\alpha_1, \cdots ,\alpha_{n-1-k}) : p(\overline{C_1}) \to 
R^{n-1-k}$, $\alpha|_{p(C_1)}$ is of class $C^1$ and 
$\| \text{grad} \,\alpha_i|_{p(C_1)}\| \le 1$ for $i=1, \cdots ,
n - 1 - k$. 

What we want to prove is that there exists $K > 0 \in R$ such that 
$$
Vol \{x'\in R^{n-1} \ | \ dist(x',C_1) \le cc'\} \le 
K Vol \{x' \in R^{n-1} \ | \ dist(x',C_1) \le c'\}
$$
for small $c' \ge 0 \in R$. 
Proceeding by induction on $k$, we can reduce the problem to 
\begin{eqnarray}\label{71}
& &  Vol \{ x' \in p^{-1}(p(C_1)) \ | \ dist(x',C_1) \le cc'\} \\
\notag
& & 
\le K Vol \{x' \in p^{-1}(p(C_1)) \ | \ dist(x',C_1) \le c'\}. 
\end{eqnarray}

\noindent By the definition of volume 
\begin{eqnarray*}
& &  Vol \{x' \in p^{-1}(p(C_1)) \ | \ dist(x',C_1) \le cc' \} \\
\notag
& & 
\le  Vol \{x' \in p^{-1}(p(C_1)) \ | \ dist(x',p(C_1)) \le cc '\}.
\end{eqnarray*}

\noindent On the other hand, since $\| \text{grad} \ \alpha_i|_{p(C_1)}\| 
\le 1$,
\begin{eqnarray*}
& &  Vol \{ x' \in p^{-1}(p(C_1)) \ | \ dist(x',p(C_1)) \le c' \} \\
\notag
& & 
\le  2^{n-1-k} Vol \{x' \in p^{-1}(p(C_1)) \ | \ dist(x',C_1)\le c'\}.
\end{eqnarray*}

\noindent Thus we can replace $C_1$ in $(\ref{71})$ with $p(C_1)$. 
Clearly $(\ref{71})$ for $p(C_1)$ holds true for $K = c^{n-1-k}$.
\end{proof}

We generalise the above lemma as follows. 

\vspace{3mm}

\noindent {\bf Lemma \ref{gctimes}$'$}\label{ggctimes}
{\em Let $A$ be a definable set-germ at $0 \in R^n$. 
Then there exists $\theta_1 \in \Phi$ such that 
$$
Vol_{ST_{c\theta}(A)} \approx Vol_{ST_\theta(A)}
$$
for $\theta \in \Phi$ and $c > 0 \in R$ with $\theta \ge \theta_1$. }

\vspace{3mm}

\begin{proof}
By Theorem \ref{eqtheorem} $A$ is $ST$-equivalent to $LD(A)$. 
Therefore there exist $\theta_2,\theta_3 \in \Phi$ such that 
$$
LD(A) \subset ST_{\theta_2}(A) \ \text{and} \ 
A \subset ST_{\theta_3}(LD(A)). 
$$
Set $\theta_1 = 2 \max (\theta_2,\theta_3)$. 
Then, as in the proof of Proposition \ref{STvolume}, we have 
$$
ST_{{\theta \over 2}}(LD(A)) \subset ST_\theta(A) \subset ST_{2\theta}(LD(A))
$$
for $\theta \in \Phi$ with $\theta \ge \theta_1$. 
Since $LD(A)$ is a definable cone with vertex $0 \in R^n$, 
by Lemma \ref{gctimes}
$$
Vol_{ST_{{\theta \over 2}}(LD(A))} \approx 
Vol_{ST_{{c\theta \over 2}}(LD(A))} \approx 
Vol_{ST_{2c\theta}(LD(A))} \approx Vol_{ST_{2\theta}(LD(A))}. 
$$
Hence 
$$
Vol_{ST_{c\theta}(A)} \approx Vol_{ST_\theta(A)}.
$$
\end{proof}

By the above arguments the following lemma is clear;
it corresponds to Lemma \ref{volumelemma}. 

\begin{lem}\label{gvolumelemma}
Let $\alpha,\beta$ be linear subspaces of $R^{n-1}$ 
with $\dim \alpha < \dim \beta$. 
Let $\alpha_1$ and $\beta_1$ denote the cones in $R^n$ with vertex 
$0 \in R^n$ and bases 
$\{x' \in \alpha \ | \ \| x' \| \le 1 \} \times \{ 1 \}$ 
and $\{ x' \in \beta \ | \ \| x' \| \le 1 \} \times \{ 1 \}$,
respectively. 
Then, for $\theta \in \Phi$, 
$$
\lim_{\epsilon\to0}\frac{Vol_{ST_\theta(\alpha_1)}(\epsilon)}
{Vol_{ST_\theta(\beta_1)}(\epsilon)} = 0. 
$$
\end{lem}

Using Lemmas \ref{gctimes} and \ref{gvolumelemma}, we can show 
the following proposition in the same way as in the proof of 
Proposition \ref{conevolume}. 

\begin{prop}\label{gconevolume}
Let $\alpha,\beta \subset R^n$ be definable cones at $0 \in R^n$. 
Suppose that $\dim \alpha < \dim \beta$. 
Then, for $\theta \in \Phi$,
$$
\lim_{\epsilon\to0}\frac{Vol_{ST_\theta(\alpha)}(\epsilon)}
{Vol_{ST_\theta(\beta)}(\epsilon)} = 0. 
$$
\end{prop}

The following lemma is clear. 

\begin{lem}\label{gdefinabledimension}
Let $A\subset R^n$ be a definable set-germ at $0 \in R^n$. 
Then we have $\dim LD(A)\le\dim A$. 
\end{lem}

We do not know whether Proposition \ref{STvolume} holds for general $R$. 
However, under some assumption, using Lemma \ref{gctimes}$'$ 
as in the proof of Proposition \ref{STvolume}, we can prove the following. 

\begin{prop}\label{gSTvolume}
Let $A,B$ be set-germs at $0 \in R^n$. 
Suppose that $A$ and $B$ are $ST$-equivalent and $A$ is definable. 
Then there exists $\theta_1 \in \Phi$ such that 
$$
Vol_{ST_\theta(A)} \approx Vol_{ST_\theta(B)}
$$
for any $\theta \in \Phi$ with $\theta \ge \theta_1$. 
\end{prop}

The following corollary is also clear. 

\begin{cor}\label{gvolumeratio}
Let $\alpha \subset R^n$ be a definable set-germ at $0 \in R^n$, 
and let $\beta \subset R^n$ be a definable cone at $0 \in R^n$. 
Suppose that $\dim \alpha < \dim \beta$. 
Then there is $\theta_1 \in \Phi$ such that 
$$
\lim_{\epsilon\to0}\frac{Vol_{ST_\theta(\alpha)}(\epsilon)}
{Vol_{ST_\theta(\beta)}(\epsilon)} = 0 
$$
for any $\theta \in \Phi$ with $\theta \ge \theta_1$. 
\end{cor}

We need to modify Lemma \ref{maintool} as follows.

\begin{lem}\label{gmaintool}
Let $h : (R^n,0) \to (R^n,0)$ be a bi-Lipschitz homeomorphism, 
and let $A$ be a definable set-germ at $0 \in R^n$. 
Suppose that $h(A)$ is definable. 
Set $E = LD(A)$ and $F = h(E)$. 
Then $\dim LD(F) \le \dim E$. 
\end{lem}

\begin{proof}
In the proof of Lemma \ref{maintool}, (6.4) was a consequence of
Lemmas \ref{osandwichlemma} and \ref{ctimes}. 
For a general $R$, (6.4) corresponds to 

\begin{eqnarray}\label{72}
Vol_{ST_\theta(F)} \approx Vol_{ST_\theta(E)} \quad \text{for} \quad 
\theta \in \Phi.
\end{eqnarray}

\noindent If $F$ is definable, (\ref{72}) follows from 
Lemmas \ref{osandwichlemma} and \ref{gctimes}$'$. 
However $F$ is not necessarily definable or we do not know 
whether (\ref{72}) holds for any $\theta$. 
We will find $\theta_1 \in \Phi$ such that (\ref{72}) holds for 
$\theta \in \Phi$ with $\theta \ge \theta_1$. 
Such a restriction does not yield any trouble in our proof.

Since $A$ and $E$ are $ST$-equivalent, by Proposition \ref{gSTvolume}
there exists $\theta_1 \in \Phi$ such that 
$$
Vol_{ST_\theta(A)} \approx Vol_{ST_\theta(E)}
$$
for any $\theta \in \Phi$ with $\theta \ge \theta_1$.
By the same reason as above and Lemma \ref{odirectionrel} we can assume that 
$$
Vol_{ST_\theta(h(A))} \approx Vol_{ST_\theta(F)}. 
$$
On the other hand, as in the proof of Lemma \ref{maintool}, 
by Lemmas \ref{osandwichlemma} and \ref{gctimes}$'$ we have
$$
Vol_{ST_\theta(A)} \approx Vol_{ST_\theta(h(A))}. 
$$
Hence (\ref{72}) holds for $\theta \ge \theta_1 \in \Phi$.

The other arguments in the proof of Lemma \ref{maintool} continue to work. 
\end{proof}

The Generalised Proposition \ref{reductions} and then Theorem 
\ref{non-Archimedean} are proved in the same way as in \S 6.

\medskip

\end{document}